\documentclass{psp}

\let\proof\relax
\let\endproof\relax
\usepackage{amsmath}
\usepackage{amsfonts}
\usepackage{mathabx}
\usepackage{verbatim}
\usepackage{amssymb}
\usepackage{amsthm}
\usepackage{stmaryrd}
\usepackage[all]{xy}
\usepackage{mathrsfs}
\usepackage{graphicx}
\usepackage{color}
\usepackage{tikz-cd}
\usepackage{tikz}
\usetikzlibrary{matrix}
\usepackage{amscd}
\usepackage{pifont}
\usetikzlibrary{arrows}
\usepackage{mathtools}
\newtheorem{thm}{Theorem}[section]
\newtheorem{theorem}[thm]{Theorem}

\newtheorem{corollary}[thm]{Corollary}
\newtheorem{lemma}[thm]{Lemma}
\newtheorem{proposition}[thm]{Proposition}

\numberwithin{equation}{section}

\theoremstyle{definition}

\newtheorem{definition}[thm]{Definition}

\newtheorem{notation}[thm]{Notation}

\numberwithin{mytheorem}{subsection}

\newtheorem{innercustomgeneric}{\customgenericname}
\providecommand{\customgenericname}{}
\newcommand{\newcustomtheorem}[2]{%
	\newenvironment{#1}[1]
	{%
		\renewcommand\customgenericname{#2}%
		\renewcommand\theinnercustomgeneric{##1}%
		\innercustomgeneric
	}
	{\endinnercustomgeneric}
}
\newcustomtheorem{prop}{Proposition}

\numberwithin{mytheorem}{subsection}
\numberwithin{myconjecture}{subsection}
\numberwithin{mydefinition}{subsection}
\numberwithin{myremark}{subsection}
\numberwithin{mysituation}{subsection}
\numberwithin{myhypothesis}{subsection}
\numberwithin{myquestion}{subsection}
\numberwithin{mynotation}{subsection}
\numberwithin{myfact}{subsection}
\numberwithin{myexamples}{subsection}
\numberwithin{myexample}{subsection}
\numberwithin{myconstruction}{subsection}
\numberwithin{mycaution}{subsection}
\numberwithin{myproposition}{subsection}
\numberwithin{mylemma}{subsection}
\numberwithin{mycorollary}{subsection}

\def\CC{\mathbb{C}}

\def\NN{\mathbb{N}}

\def\QQ{\mathbb{Q}}
\def\RR{\mathbb{R}}
\def\SS{\mathbb{S}}

\def\UU{\mathbb{U}}

\def\ZZ{\mathbb{Z}}

\def\calZ{\mathcal{Z}}

\def\vp{\mathrm{val}_p}

\newcommand{\f}{\frac}

\begin{document}

		\title[Primitive prime divisors in the critical orbits of rational polynomials]
		{Primitive prime divisors in the critical orbits of one-parameter families of rational polynomials}
		\author[Rufei Ren]{RUFEI REN\\
			Department of Mathematics, Fudan University,\\220 Handan Rd., Yangpu District, Shanghai 200433, China \addressbreak
			e-mail\textup{: \texttt{rufeir@fudan.edu.cn}}}

\receivedline{Received \textup{27} October \textup{2020}}

\date{\today}

%\keywords{Primitive divisors, Zsigmondy set, critical orbit, polynomials}
%\subjclass[2010]{11F41, 37F10, 37P05}
\maketitle

	\begin{abstract}
	For a  polynomial $f(x)\in\QQ[x]$ and rational numbers $c, u$, we put $f_c(x)\coloneqq f(x)+c$, and consider  the Zsigmondy set $\calZ(f_c,u)$ associated to the sequence $\{f_c^n(u)-u\}_{n\geq 1}$, see Definition~\ref{Zsig}, where $f_c^n$ is the $n$-st iteration of $f_c$. 
	In this paper, we prove that if $u$ is a rational critical point of $f$, then there exists an $\mathbf M_f>0$ such that  $\mathbf M_f\geq \max_{c\in \QQ}\{\#\calZ(f_c,u)\}$.
\end{abstract}

\tableofcontents

%\tableofcontents

\section{Introduction}\label{sec:intro}

For every polynomial $f(x)\in \QQ[x]$ and  $\alpha\in \QQ$ we put $f_\alpha(x)\coloneqq f(x)+\alpha.$
Therefore, $f_\alpha$ can be considered as a one-parameter family of polynomials. For every $u\in \QQ$ we write \[\SS_{f,u}\coloneqq \{c\in \QQ\mid \{f_c^n(u)-u\}_{n\geq1} \textrm{~is infinite}\},\]
where $f_c^n$ is the $n$-st iteration of $f_c$. 
In particular, if $u=0$, we put $\SS_f\coloneqq\SS_{f,0}$.

We denote by $\vp(-)$ the $p$-adic valuation of $\QQ$ normalized by $\vp(p)=1$.
In keeping with the terminology of \cite{IngramSilverman2009}, for every polynomial $f(x)\in \QQ[x]$, $u\in \QQ$  and $n\geq 1$ we say that
$p$ is a {\bf primitive prime divisor} of $f^{n}(u)-u$ if
$\vp(f^{n}(u)-u) > 0$ and $\vp(f^{k}(u)-u) \leq 0$
for all $1\leq k < n$.

\begin{definition}\label{Zsig}
	The \textbf{Zsigmondy set} of the sequence $\{f^n(u)-u\}_{n\geq1}$ is defined by
	\begin{equation*}
	\calZ(f,u)\coloneqq \{n\geq 1\;|\; f^n(u)-u \textrm{~has no primitive prime divisor}\}.
	\end{equation*}
\end{definition}

The
primary application of bounds on the Zsigmondy set is towards understanding arboreal Galois representations associated to iteration of rational maps over number
fields. It is first studied by Bang \cite{Bang1886} and Zsigmondy \cite{Zsigmondy1892}. Since then, there have been quite a few research papers  on characterizing/bounding Zsigmondy sets of various sequences in various settings, e.g., Carmichael \cite{Carmichael1913}, Schinzel \cite{Schinzel1974}, Rice \cite{Rice2007}, Ingram--Silverman \cite{IngramSilverman2009}, Doerksen--Haensch \cite{DoerksenHaensch2011}, Gratton--Nguyen--Tucker\cite{GrattonNguyenTucker2013}, Krieger \cite{Krieger2013}, etc.

In this work, we are interested in the size of the Zsigmondy set of a sequence obtained from the critical orbit of polynomials with rational coefficients of degree $d\geq 2$. We denote by $\Sigma$ the set of finite primes of $\ZZ$ and reserve $p$ for a prime number.

We first state our main theorem.

\begin{theorem}\label{main thm3}
	For every polynomial $f\in \QQ[x]$ of degree $d\geq2$ with a critical point $u\in \QQ$ there is a constant $\mathbf{M}_f>0$, depending only on $f$ (independent of $c\in \QQ$), such that
\[\#\calZ(f_c,u)\leq \mathbf{M}_f\]  for every $c\in \SS_{f,u}$.

\end{theorem}

It is worth mentioning that Rice \cite{Rice2007} was the first to prove the finiteness of $\calZ(f,0)$ for each individual polynomial $f(x)\neq x^d$.
In \cite{DoerksenHaensch2011}, Doerksen--Haensch  prove Theorem~\ref{main thm3}
for the case that $f(x)=x^d$, $u=0$ and    $c\in\ZZ$, which is generalized by Krieger in 
\cite{Krieger2013} to every $c\in \QQ$, see \cite[Theorem 1.1]{Krieger2013}.
Our contribution is to prove Theorem~\ref{main thm3} for general polynomials which are not necessary to be monic nor integer. As we consider polynomials $f$ that are more complicated than $x^d$, we did not aim to get the sharpest uniform bound $\mathbf{M}_f$.

\begin{definition}	
	A polynomial $g(x)\in\QQ[x]$  is called \textbf{$x^2$-divisible} if it has degree $d\geq2$ and is of the form \[g(x)=u_dx^d+\cdots+u_2x^2\in\QQ[x].\]
\end{definition}

At the last section we will prove that the following theorem implies Theorem~\ref{main thm3}.

\begin{theorem}\label{main thm 2}
	Given an  $x^2$-divisible  $g(x)\in\ZZ[x]$  of degree $d\geq3$ there is a constant $\mathbf{M}_g>0$, depending only on $g$, such that
	\[\#\calZ(g_c,0)\leq \mathbf{M}_g\]  for every $c\in \SS_{g}$.
\end{theorem}
%------------------------

Note that one can give an explicit expression of the lower bound $\mathbf{M}_g$ when combining the decomposition of $\SS_g$ in Proposition~\ref{thm:fund ineq} with Propositions~\ref{P:1}, \ref{P:2}, \ref{P:3} and \ref{P:4}.

%\begin{remark}
%    One can think of $L_g$ as an indicator of the complexity of the monic polynomial $f$. Then our theorem says that as $L_g\rightarrow\infty$ the order of growth of for $\#\calZ(f_c,0)$ is $\ll L_g$ uniformly for $c\in\ZZ$, and more interestingly, is $\ll\log\ln L_g$ uniformly for $c\in\QQ-\ZZ$. It is clear that many of the results in \S2 could be instantly improved if one considers the dependence on $\deg f=d$.
%\end{remark}
%
%\noindent As the reader will see, our assumptions on $f$ make the dynamics of $f_c$ more tractable. It would be interesting to see if this is true for other classes of polynomials, which might require new ideas and methods.
%\vskip 1ex
%\begin{proposition}\cite[Theorem~1.1]{Krieger2013}
%	Theorem~\ref{main thm} holds for $d=2$.
%\end{proposition}
%Therefore, we may assume that $d\geq 3$.

This paper is inspired by Krieger's work in \cite{Krieger2013}. We generalize her result from the special polynomial $f(x)=x^d$ to arbitrary polynomials in $\QQ[x]$. We first address the difficulties  on this generalization as follows.

%Moreover, for each germ one can write the uniform upper bound of $\#\calZ(c,0)$ as a function of $L_g$ and $u_d$ by tracing the proof in this paper.

The first difficulty is from dealing with the non-monic case, in which the denominator $f_c^n(c)$ is no longer always equal
to $d^n$'s power of the denominator of $c$.
To conquer it, we introduce a factorization of an integer with respect to the leading term $u_d$ of $f$, see (\ref{Eq1}), which allows us to focus on the major factor of the denominator of  $f_c^n(c)$. 

The second difficulty is from the critical points of large multiplicities. Due to this reason, some arguments in \cite{Krieger2013} do not work for our case. For example, Krieger uses Mahler's theorem to control $|f_c^n(c)|$ by $|f_c^{n-1}(c)-f_c^{-1}(0)|$. However, this estimation might not be enough when $f_c^{n-1}(c)$ is very close to a critical point with large multiplicity.
It forces us to control $|f_c^n(c)|$ in Proposition~\ref{p:1} by $|f_c^{n-N}(c)-f_c^{-N}(0)|$ for some relatively large $N>1$.

\subsection*{Acknowledgment}
The author would like to thank Tom Tucker and Shenhui Liu for their valuable discussions.

\section{Introduction of Proposition~\ref{thm:fund ineq} and some estimates}\label{s2}
We split this section into two parts. In the first part, we introduce our main technical result Proposition~\ref{thm:fund ineq} whose proof will be given in \S\ref{s32}, and prove that it implies  Theorem~\ref{main thm 2}. In the second part, we focus on estimating $\ln |A_{n}|$ which appears in Proposition~\ref{thm:fund ineq}.

Let us first set conventions and introduce some notations.
\begin{enumerate}
	  \renewcommand{\theenumi}{(\arabic{enumi})}
	\item We set $\NN\coloneqq \{1,2,\dots\}$ to be the set of natural numbers and for every $n\in \NN$ we denote by $[n]$ the finite set $\{1,2,\dots,n\}$.
	\item We denote by $\Sigma$ the set of finite primes of $\ZZ$ and reserve $p$ for a prime number. For every $n\in \ZZ\backslash\{0\}$ the sum $\sum_{p|n}$ and the product $\prod_{p|n}$ are taken over all its distinct prime factors whose number is denoted by $\omega(n)$.
\end{enumerate}
We will always write an $x^2$-divisible $g(x)\in \ZZ[x]$ by $g(x)=u_dx^d+\cdots+u_2x^2\in\ZZ[x]$, and define its \textbf{length} by \[L_g\coloneqq 1+\sum_{i=2}^{d-1}|u_i|/|u_d|. \]

For every $x^2$-divisible  $g(x)\in \ZZ[x]$, $c\in \QQ$ and $n\geq0$ we write the $(n+1)$-st iteration $g_c^{n+1}(0)$ as
\[
g_c^{n+1}(0)=g^n_c(c)\coloneqq \frac{A_{n}}{B_{n}},
\]
where $B_{n}>0, A_n$ are coprime and both depend on $g$ and $c$.
Clearly, we have $c=g_c(0)=\f{A_{0}}{B_{0}}$.

%For $c=a/b\in\QQ$, we work with $g_c^{n+1}(0)=g_c^n(c)$ and write $b=\prod_{p}p^{R_{p}}$ and
%$$
%g_c^{(n)}(c)=\frac{A_n}{B_n}\quad\mbox{with}\quad \gcd(A_n,B_n)=1.
%$$

\begin{definition}
	For an $x^2$-divisible  $g(x)\in \ZZ[x]$ and a set $S$ in $\SS_g$ we call that $g$ has \textbf{rapidly increasing numerators} on $S$ if there exists an integer $N>0$ such that 
	for every $c\in S$  there is a finite set $J_c$ with $\#J_c\leq N$ such that for every $n\notin J_c$ we have
	\begin{equation}\label{eq:2}
	\ln |A_{n}|> \sum_{p|n}\ln \left|A_{n/p}\right|.
	\end{equation}
\end{definition}

We now state our main proposition, which is followed by the proof of Theorem~\ref{main thm 2}.
%To bound $\calZ(g_c,0)$, we establish the following fundamental inequality.
\begin{proposition}[Main Proposition]\label{thm:fund ineq}
	Every $x^2$-divisible  $g(x)\in \ZZ[x]$ has rapidly increasing numerators on $\SS_g$.
\end{proposition}

\begin{proof}[Proof of Theorem~\ref{main thm 2} in assuming proposition~\ref{thm:fund ineq}]
	By \cite[Lemma~2.3 and Corollary~2.4]{Krieger2013}, if $n \in \calZ(g_c,0)$, then
	$A_n \;|\;\prod_{p|n}A_{n/p}
	$
	and hence
	\begin{equation*}
	\ln |A_{n}|\leq \sum_{p|n}\ln \left|A_{n/p}\right|.
	\end{equation*}
	Together with Proposition~\ref{thm:fund ineq}, this finishes the proof.
\end{proof}

%\begin{definition}
%	We call a pair of reduced rational numbers $\frac{\alpha}{\beta}$ and
%	$c=a/b$ \textbf{compatible} if they satisfies
%	\begin{itemize}
%		\item  $\ord_{q_j}(b)<ds_j$ for any $1\leq j\leq m$,
%		\item $\ord_{p_i}(b)<(d-1)R_i$ for any $i\in I_\beta$,
%		\item $\ord_{p_i}(b)\leq R_i$ for any $i\notin I_\beta$.
%	\end{itemize}
%\end{definition}

The naive idea of proving Proposition~\ref{thm:fund ineq} is to give a lower bound for $\ln|A_n|$ and an upper bound for $\prod_{p|n}|A_{n/p}|$ such that the lower bound is always greater than the upper bound when $n$ is large enough. Consider that \begin{equation}\label{eqq:1}
\ln |A_{n}|=\ln B_{n}+\ln |g_c^n(c)|.
\end{equation} It is sufficient for us to control $\ln B_{n}$ and $\ln |g_c^n(c)|$.

%
%
%The strategy of proving Proposition~\ref{thm:fund ineq} is that for general $n\geq 1$ 
%we provide an upper and lower bound for $\ln |A_{n}|$ as functions of $f$, $c$ and $n$. The upper bound will control the right hand side of \eqref{eq:2}, while the upper one controls the left hand side. Moreover, consider that 

\subsection{Upper bounds for $\ln B_n$ and $\ln |g_c^n(c)|$}

\begin{lemma}\label{lem:Bn<}
Given any $x^2$-divisible $g(x)\in \ZZ[x]$ and $c\in \QQ$,	for every $n\geq 0$ we have
	%\begin{equation}\label{Bn<}
	\begin{enumerate}
		  \renewcommand{\theenumi}{(\arabic{enumi})}
		\item $\ln B_{n}\leq  {d^n}\ln B_0;$
		\item $	\ln|g_c^{n}(c)|\leq d^{n}\ln \Big(2|u_d|\max\left\{|c|,4L_g\right\}\Big).$
	\end{enumerate}
	%\end{equation}	
\end{lemma}
\begin{proof}
	(1) Since $g_c^n(c)$ can be written as $A_n'/B_0^{d^{n}}$ for some $A_n'\in \ZZ$, we have
	$B_{n}| B_0^{d^n}$.
	
	(2)	It is enough to prove 
	\begin{equation*}
	|g_c^{n}(c)|\leq (2|u_d|)^{\frac{d^n-1}{d-1}} \Big(\max\left\{|c|,4L_g\right\}\Big)^{d^{n}}\quad \textrm{for all~}n\geq 0.
	\end{equation*}	
	For $n=0$, we have $	|g_c^{0}(c)|=|c|\leq \!\max\left\{|c|,4L_g\right\}.$
	
	Assume that the desired inequality holds for some $n\geq 0$ and temporarily denote its right side by $T_{n}$. Then we have
	\begin{multline*}
	|g^{{n}+1}_c(c)|\leq|u_d||g_c^{n}(c)|^d+\sum_{i=2}^{d-1}|u_i||g_c^{n}(c)|^i+|c|\leq |u_d|(T_{n}+L_g) T_{n}^{d-1}\\
	\leq |u_d|\cdot 2T_{n} \cdot T_{n}^{d-1}=
	(2|u_d|)^{\frac{d^{n+1}-1}{d-1}} \Big(\max\left\{|c|,4L_g\right\}\Big)^{d^{n+1}}.
	\end{multline*}
	The proof follows by induction.
\end{proof}

\subsection{A lower bound for $\ln B_n$}
Consider that 	\begin{equation}
\label{AkBk}\f{A_{n+1}}{B_{n+1}}=g_c^{n+1}(c)=g_c\!\left(\f{A_{n}}{B_{n}}\right)=\sum_{i=2}^d u_i\frac{A_{n}^i}{B_n^i}+ \frac{A_0}{B_0}.
\end{equation}

\begin{lemma}\label{l1:1}
Given any $x^2$-divisible $g(x)\in \ZZ[x]$ and $c\in \QQ$,	for every $n\geq 0$ if $p\in I(B_n)$, then we have
	\begin{enumerate}
		\item $p\in I(B_{n+1})$ and
		\item $\vp(B_{n+1})=d\vp(B_n)-\vp(u_d).$
	\end{enumerate}
\end{lemma}
\begin{proof}
	Note that for every $n\geq 0$ we have
	\begin{align}\label{eqq6}
	\vp\left(\frac{u_dA_{n}^d}{B_n^d}\right)&=\vp(u_d)+d\vp(A_n)-d\vp(B_n),\\
	\vp\left(\frac{u_iA_{n}^i}{B_n^i}\right) &\geq (1-d)\vp(B_n),\quad i=2,\ldots,d-1, \\
	\vp(\f{A_0}{B_0}) &\geq -\vp(B_0).
	\end{align}
	Therefore, if $p\in I(B_n)$, then we have
	$\vp(A_n)=0$	and 
	\begin{equation}\label{eq1}
	\vp\left(\frac{u_dA_{n}^d}{B_n^d}\right)<(1-d)\vp(B_n)\leq		\vp\left(\frac{u_iA_{n}^i}{B_n^i}\right)
	\end{equation}
	for every $2\leq i\leq d-1$. Note that the term on the right hand side of this inequality does not exist for the case $d=2$.
	
	We now prove this lemma by induction. 
	
	For $n=0$ we have $\vp(B_0)>\vp(u_d)$. Combined with (\ref{AkBk}) and (\ref{eq1}), this implies
	\[\vp\left(\f{A_{1}}{B_{1}}\right)=	\vp\left(\frac{u_dA_{0}^d}{B_0^d}\right)=\vp(u_d)-d\vp(B_0)<-\vp(u_d)\]
	and hence $p\in I(B_1)$.
	
	Now we assume that this lemma holds for every $0\leq k\leq n$. 
	
	(1) If $p\not\in I(B_0)$, we have \[\vp(B_0)\leq \vp(u_d)<\vp(B_n).\] Combined with  (\ref{AkBk}) and (\ref{eq1}), this implies
	\begin{equation}\label{eq2}
	\vp\left(\f{A_{n+1}}{B_{n+1}}\right)=	\vp\left(\frac{u_dA_{n}^d}{B_n^d}\right)=\vp(u_d)-d\vp(B_n)<-\vp(u_d)
	\end{equation}
	and hence $p\in I(B_{n+1})$.
	
	(2) 	If $p\in I(B_0)$, then by induction, we have $p\in I(B_n)$ and \[\vp(B_n)=d^{n}\vp(B_0)-\vp(u_d)\sum_{i=0}^{n-1}d^i\geq \vp(B_0).\] 
	Combining (\ref{AkBk}) with (\ref{eq1}), we also obtain (\ref{eq2}). This completes the induction. 
\end{proof}

For every $a\in \ZZ$ 
we denote \[I(a)\coloneqq \{p\in \Sigma\mid \vp(a)>\vp(u_d)\},\] and put
\begin{equation}\label{Eq1}
\widehat a\coloneqq\prod\limits_{p\in I(a)}p^{\vp(a)}.
\end{equation} 
When $I(a)$ is empty, we put $\widehat{a} \coloneqq 1$.
Note that we always have 
\begin{equation}\label{Eq2}
|a|\leq |u_d\widehat{a}|.
\end{equation}

\begin{lemma}\label{first}
Given any $x^2$-divisible $g(x)\in \ZZ[x]$ of degree $d\geq 3$ and $c\in \QQ$,	for every $0\leq n'\leq n$ we have
	\[\ln B_{n}\geq \frac{d^{n-n'}}{3}\ln \widehat{B_{n'}}.\]
\end{lemma}

\begin{proof}
	If $\widehat{B_{n'}}=1$, it is trivial.
	
	Now we assume that $\widehat{B_{n'}}\geq 2$. It is enough to show that every
	prime $p\in I(B_{n'})$ satisfies 
	\begin{equation}\label{eqq1}
	\vp(B_{n})\geq \frac{d^{n-n'}}{3}\vp\left(B_{n'}\right).
	\end{equation}

	Using Lemma~\ref{l1:1} inductively, we have
	\begin{equation*}
	\vp\left(B_{n}\right) =d^{n-n'}\vp(B_{n'})-\vp(u_d)\sum_{i=0}^{n-n'-1}d^i\geq \left(d^{n-n'}-\sum_{i=0}^{n-n'-1}d^i\right)\vp(B_{n'}).
	\end{equation*}
	
	From our assumption that $d\geq 3$, we have \[d^{n-n'}-\sum\limits_{i=0}^{n-n'-1}d^i\geq d^{n-n'}-2d^{n-n'-1}\geq \frac{1}{3}d^{n-n'}, \]
	which completes the proof.
\end{proof}

\subsection{The lower bound for $\ln |g_c^n(c)|$}
\begin{lemma}\label{lemma f>1}
	Given any $x^2$-divisible $g(x)\in \ZZ[x]$ and $c\in \QQ$, if $|g_c^{n'}(c)|\geq\max\left\{4L_g,|c|\right\}$ for some $n'\geq 0$, then for every $n\geq n'$ we have
	\begin{equation*}
	|g_c^{n}(c)|\geq 2^{-\frac{d^{n-n'}-1}{d-1}}\cdot |g_c^{n'}(c)|^{d^{n-n'}} .
	\end{equation*}	
	Clearly, in this case $c$ is in the basin of infinity for $g_c$.
\end{lemma}
\begin{proof}
	
	The proof follows from induction.
	For $n=n'$ this lemma is trivial. 
	
	Assume that  this lemma holds for some $n\geq n'$. Then we have 
	\[|g_c^{n}(c)|\geq 2^{-\frac{d^{n-n'}-1}{d-1}}\cdot|g_c^{n'}(c)|^{d^{n-n'}}\geq \max\{4L_g, |c|\}\cdot \left|4L_g/2\right|^{\frac{d^{n-n'}-1}{d-1}}\geq \max\{4L_g, |c|\},\]
	and hence
	\begin{multline}\label{eq:11}
	|g_c^{n+1}(c)|\geq |u_d||g_c^{n}(c)|^d-\sum_{i=2}^d|u_i||g_c^{n}(c)|^{i}-|c|
	\\\geq |u_d|\Big(|g_c^{n}(c)|^d-L_g|g_c^{n}(c)|^{d-1}\Big)
	= |g_c^{n}(c)|^{d-1}|u_d|\Big(|g_c^{n}(c)|-L_g\Big)
	\\\geq   |g_c^{n}(c)|^{d}/2\geq 2^{-\frac{d^{n-n'+1}-1}{d-1}}\cdot |g_c^{n'}(c)|^{d^{n-n'+1}} .\qedhere
	\end{multline}
	%	By mathematical induction, we have
	%	$$g_c^{n}(\frac{\alpha}{\beta})\geq \max\{|c|,4L_g\}. $$ It is followed by
	%	$$\frac{3}{4}(g_c^{n-1}(\frac{\alpha}{\beta}))^{d}\leq g_c^{n}(\frac{\alpha}{\beta})\leq \frac{5}{4}(g_c^{n-1}(\frac{\alpha}{\beta}))^{d}.$$
	%	It implies
	%	$$|g_c^{n}(\frac{\alpha}{\beta})|\geq |c|^{d^n} (\frac{3}{4})^{\sum\limits_{i=0}^{n-1} d^i}>|c|^{d^n}(\frac{3}{4})^{\frac{d^n}{d-1}}$$
	%	and
	%	$$|g_c^{n}(\frac{\alpha}{\beta})|\leq |c|^{d^n} (\frac{5}{4})^{\sum\limits_{i=0}^{n-1} d^i}<|c|^{d^n}(\frac{3}{4})^{\frac{d^n}{d-1}}$$
	%	%	From the condition that $M>>0$, we know that $$g_c^{n}(a/b)\geq (\frac{1}{2})^{d^{n-1}+\dots+1}M^{d^n}.$$
\end{proof}

\begin{corollary}\label{lemma f>}
Given any $x^2$-divisible $g(x)\in \ZZ[x]$ and $c\in \QQ$,	if $|g_c^{n'}(c)|\geq\max\left\{4L_g,|c|\right\}$ for some $n'\geq 0$, then for every $n\geq n'$ we have
	\begin{equation*}
	\ln |g_c^{n}(c)|\geq d^{n-n'}\ln |g_c^{n'}(c)/2|.
	\end{equation*}	
\end{corollary}
\begin{proof}
	It follows directly from Lemma~\ref{lemma f>1}.
\end{proof}

Given an algebraic number $\gamma\in \CC$ of degree $\ell$ with conjugates $\gamma_1\coloneqq \gamma,\gamma_2,\dots,\gamma_\ell$ over $\QQ$, let $a_0$ be an integer such that the coefficients of the polynomial $g(X)=a_0\prod_{i=1}^{\ell}(X-\gamma_i)$ are integers of $\gcd\ 1$, then 
we define the Mahler measure of $\gamma$ by
\[M(\gamma)\coloneqq |a_0|\prod_{i=1}^{\ell}\max(1,|\gamma_i|).\] 

%\begin{lemma}
%	For $c\in \QQ$ and $\zeta\in\CC$ such that $|c|<2N_f$ and $|z|<2N_f$, any root $\gamma$ of $g_c-\zeta$ satisfies
%	\begin{equation}\label{coe}
%	|\gamma|\leq 2N_f.
%	\end{equation}
%\end{lemma}
%\begin{proof}
%	\Ren{need to modify it slightly}Cauchy's bound on roots of polynomials and the size of $c$ imply that the roots of $g_c$ lie in the disc $|z|\leq\Delta+1$. 
%\end{proof}
%\begin{corollary}
%	For any $n\geq 1$, if $|c|\leq 2N_f$, then we have all the roots of $g_c^n(x)=0$ have norm less than $2N_f$.
%\end{corollary}
%\begin{proof}
%	Follows from mathematical induction.
%\end{proof}

\begin{notation}
	For every $r\geq 1$ and $\delta>0$ we put 	\[ W(r,\delta)\coloneqq 2 \times 10^7 \delta^{-4} \cdot\ln 4r\cdot\ln \ln 4r.\]
\end{notation}

\begin{theorem}[\cite{Evertse1997}, Theorem 1]\label{ever}
	Let $0 < \delta < 1.$
	Then for every algebraic number $\gamma$ of degree $r\geq 1$, there are at most 	$W(r,\delta)$
	solutions $a/b\in \QQ$ to 
	\begin{equation}
	\label{ineq:diophantine}
	|a/b-\gamma| < M(a/b)^{-2-\delta}	\end{equation} 
	with
	$M(a/b)\geq \max\{4^{2/\delta},M(\gamma)\}.$
\end{theorem}

Theorem~\ref{ever} implies the following result.

\begin{corollary}\label{ever1}
Given any $x^2$-divisible $g(x)\in \ZZ[x]$, for every $L>0$, $N\in \NN$, $c\in [-L,L]\cap \QQ$ and $\gamma\in \CC$ such that $g^{N}_c(\gamma)=0$, there is an integer $D>0$, independent of $c$, such that 
	\begin{equation}
	|a/b-\gamma|<\frac{1}{(2bD)^3}
	\end{equation}
	has at most
	$W(d^{N},1/10)$
	rational solutions $a/b$ with $b \geq \max\left\{4^{20}, (|u_d|B_{ 0}D)^{d^N}\right\}.$
\end{corollary}
\begin{proof}
	Let $\gamma_{c,1},\dots,\gamma_{c,d^N}$ be the roots of $g_c^N(x)=0$ in $\CC$ which are not necessary to be distinct.
	Since $g_c^{N}({x})$ is continuous as a function of $x$ and $c$, there exists an integer $D>1$ such that for every $c\in [-L,L]\cap \QQ$ and $1\leq i\leq d^N$ we have \begin{equation}\label{Eq:9}
	|\gamma_{c,i}|<D.
	\end{equation} 	Without loss of generality, we put $\gamma\coloneqq \gamma_{c,1}$ and  $h(x)\coloneqq a_0\prod_{i=1}^\ell(x-\gamma_{c,i})$ to be the minimal polynomial of $\gamma$ with integer coefficients of $\gcd \ 1$. 
	
	Since $B_{ 0}^{d^N}g^{N}_c(\gamma)=0$ and $B_{ 0}^{d^N}g^{N}_c(x)$ is a polynomial with integer coefficients, we have $h(x)|B_{ 0}^{d^N}g^{N}_c(x)$. Combined with Gauss's lemma, this implies 
	\begin{equation}\label{Eq3}
	a_0|(u_dB_{ 0})^{d^N}.
	\end{equation} 
	
	Combining (\ref{Eq:9}) with (\ref{Eq3}), we have 
	\begin{equation}\label{EQ1}
	M(\gamma)<(|u_d|B_{ 0})^{d^N}D^\ell.
	\end{equation}
	
	On the other hand, for every rational number $a/b$ in the lowest terms such that $|a/b|\leq 2D$ we have
\begin{equation}\label{new1}
	b\leq M(a/b)\leq 2bD.
\end{equation}
	
		Note that Theorem~\ref{ever} still holds when we do the following modifications.
	\begin{enumerate}
		   \renewcommand{\theenumi}{(\arabic{enumi})}
		\item Restricting this theorem to a set of algebraic numbers and changing $M(\gamma)$ to a function of $\gamma$ which is larger than $M(\gamma)$ for every $\gamma$ in this set.
		\item Changing the right hand side of (\ref{ineq:diophantine}) to a function of $a/b$ which is less than $M(a/b)^{-2-\delta}$ for every rational number $a/b$.
		\item  Changing the second $M(a/b)$ in Theorem~\ref{ever} to a function of $a/b$ which is less than $M(a/b)$ for every rational number $a/b$.
	\end{enumerate}
	
	Therefore, combined with (\ref{EQ1}) and (\ref{new1}), 
	 Theorem~\ref{ever} implies that there are at most 	$W(\ell,\delta)$
	rational  solutions $a/b$ to
	\begin{equation}\label{Eq4}
	|a/b-\gamma| < \left(2bD\right)^{-2-\delta}	\end{equation} 
	such that $|a/b|\leq 2D$ and
	$b\geq \max\{4^{2/\delta},(|u_d|B_{ 0})^{d^N}D^\ell\}.$
	
	For  rational number $a/b$ such that $|a/b|\geq 2D$ we have
	\[|a/b-\gamma|\geq D>1>(2bD)^{-2-\delta}.\]
	
	Together with (\ref{Eq4}), this shows that
	there at most 	$W(\ell,\delta)$
	rational  solutions $a/b\in \QQ$ to
	\begin{equation}\label{eqqq}
|a/b-\gamma| < \left(2bD\right)^{-2-\delta}	\end{equation} 
	with
	$b\geq \max\{4^{2/\delta},(|u_d|B_{ 0})^{d^N}D^\ell\}.$
	
	Take $\delta\coloneqq 1/10$. 
	Combining $W(\ell,1/10)\leq  W(d^N,1/10)$ with   the modification(3) above, we can replace $\ell$ by $d^N$ and $-2-1/10$ by $-3$, which completes the proof.	
\end{proof}

Now we consider $|c|\leq 4L_g$.  Recall that $\SS_{g}$ is the set that contains all rational number $c$ such that $\{g_c^n(0)\}$ is infinite. By Corollary~\ref{lemma f>} with $n'=0$, for every $c\in (-\infty,-4L_g]\cup[4L_g,\infty)$ we have $\lim_{n\to \infty} \ln |g_c^{n}(c)|=\infty$ and hence
$\SS_{g}\supset (-\infty,-4L_g]\cup[4L_g,\infty)$.

We denote by $\UU_g^0$ the finite subset of $\SS_{g}\cap[-4L_g,4L_g]$ consisting of all the rational numbers with denominator dividing $u_d$ and put $\UU_g\coloneqq [-4L_g,4L_g]\cap(\SS_g\backslash \UU_g^0)$. 
 The following proposition aims at dealing the case $c\in \UU_g$. It is worth noting that $\widehat{B_{0}}\geq 2$
 for all $c\in \UU_g$.

\begin{proposition}\label{p:1}
	For an $x^2$-divisible  $g(x)\in \ZZ[x]$ and a real number $\alpha\in[-4L_g,4L_g]$ such that $g(x)\neq u_dx^d$  or $\alpha\neq 0$, there exists $0<\delta<L_g$, $C>0$ and an integer $N\geq 0$ such that for every $c\in (\alpha-\delta,\alpha+\delta)\cap \SS_g$ if $\widehat{B_{n'}}\geq 2$ for some $n'\geq 0$, then
	there is a finite set $S_c\subset \NN$ of bounded cardinality $N+n'$ such that for every $n\not\in S_c$ we have
	\[\ln|g_{c}^{n}(c)|\geq \min\left\{({-1+1/d})\ln B_n+\ln C,\ \ln \delta\right\}.\]	
\end{proposition}

\begin{proof}
	%	
	%Now we inductively label the points in set $g_c^{-k}(0)$ for $1\leq k\leq N_0.$ For every $1\leq k\leq N_0-1$, $j\in [d]$ and  every sequence $\ui\coloneqq (i_1,i_2,\dots, i_k)\in [d]^{k}$ we put $(\ui,j)\coloneqq (i_1,\dots,i_k,j)$.
	%
	
	Let \begin{equation*}
	N_0\coloneqq \left\lceil\frac{2\ln 3}{\ln (d/(d-1))}\right\rceil+1,
	\end{equation*}
	which satisfies
	\begin{equation}\label{eq:n0}
	9(d-1)^{N_0-1}\leq d^{N_0-1}.
	\end{equation}
	
	Let $\gamma_1,\dots,\gamma_r\in \CC$ be the distinct roots of $g^{N_0}_\alpha(x)=0$ of multiplicity $m_1,\dots,m_r$, respectively. 
	Choose an $0<\epsilon<1$ small enough such that for any two distinct $i,j\in [r]$ we have 
	$|\gamma_i-\gamma_j|> 3\epsilon$.
	
	By continuity of $g^N_c(x)$ as a function of $x$ and $c$, there exists $0<{\delta}<L_g$ such that for every $1\leq i\leq r$ and $\alpha', \beta\in \RR$ with $|\alpha'-\alpha|<\delta$ and $|\beta|<\delta$ there are exactly $m_i$ roots of $g_{\alpha'}^{N_0}(x)-\beta=0$ in the disk
	$O(\gamma_{i},\epsilon)\subset \CC$.
	
	%the preimage $g_{c}^{-k}(O(0,{\delta}))$ is a disjoint union of open sets labeled as $\{O_{\ui}\}_{\ui\in S_k}$ which satisfies that
	%\begin{itemize}
	%	\item for every $\ui\in \bigsqcup\limits_{k=0}^{N_0} S_k$ we have
	%	$O_{\ui}\subset O(\gamma_{\ui},\epsilon)$.
	%	\item For every $\ui\in \bigsqcup\limits_{k=0}^{N_0-1} S_k$ we have
	%	$g^{-1}(O_{\ui})=\bigsqcup\limits_{j=1}^{r_{\ui}}O_{(\ui,j)}$; and each 
	%	$0\leq j\leq r_{\ui}$ the restricted map $g_{c}: O_{(\ui,j)}\to O_{\ui}$ is onto.
	%\end{itemize} 
	Now we consider an arbitrary $c\in (\alpha-\delta,\alpha+\delta)\cap \SS_g$.
	
	Let $\Gamma$ be the multiset consisting of all the roots of $g_{c}^{N_0}(x)=0$, i.e. two elements in $\Gamma$ could be the same.
	From the argument above, for every $n\geq N_0$ if $|g_{c}^{n}(c)|< \delta$, then there exists $1\leq i_0\leq r$ such that $g_c^{n-N_0}(c)\in O(\gamma_{i_0},\epsilon)$.
	We put \begin{equation*}
	\Gamma_1\coloneqq \Gamma\backslash O(\gamma_{i_0},\epsilon)\quad\textrm{and}\quad
	\Gamma_2\coloneqq \Gamma\cap O(\gamma_{i_0},\epsilon).
	\end{equation*} 
	Note that we have \begin{equation}\label{eq:key0}
	\Gamma=\Gamma_1\cup \Gamma_2,\ \#(\Gamma)=d^{N_0} \textrm{~and~} \#(\Gamma_2)=m_{i_0}.
	\end{equation}

	Now we count the distance between $g_c^{n-N_0}(c)$ and the points in $\Gamma$.
	
	For every $\xi\in \Gamma_1 $, from our choice of $\epsilon$, we have
	\begin{equation}\label{eq:key1}
	\left|\xi-g_c^{n-N_0}(c)\right|> 3\epsilon-2\epsilon=\epsilon.
	\end{equation}

	For every $\xi\in \Gamma_2$, by Corollary~\ref{ever1} with $L\coloneqq 4L_g$ and $N\coloneqq N_0$, there is an integer $D>0$, independent of $c$, such that 
	\begin{equation}\label{Eq6}
	|a/b-\xi|<\frac{1}{(2bD)^3}
	\end{equation}
	has at most
	$W(d^{N_0},1/10)$
	rational solutions $a/b$ with $b \geq \max\{4^{20}, (|u_d|B_{ 0}D)^{d^{N_0}}\}.$ 
	
	Put \[N_1\coloneqq \left\lceil\log_d 120\right\rceil+N_0 \quad\textrm{and}\quad N_2\coloneqq \left\lceil 3\log_d log_2 (u_d^2D+1) \right\rceil+2N_0.\]  Then for every $n\geq N_1$, by Lemma~\ref{first}, we have
	\[\ln B_{n+n'-N_0}\geq \frac{d^{n-N_0}}{3}\ln \widehat{B_{n'}}\geq \frac{d^{n-N_0}}{3}\ln 2\geq 20\ln 4.\]
	
	If $\widehat{B_{0}}=1$, by Lemma~\ref{first}, (\ref{Eq2}) and the choice of $N_2$, for every $n\geq N_2$ we have
	\[\ln B_{n+n'-N_0}\geq \frac{d^{n-N_0}}{3}\ln \widehat{B_{n'}}\geq  \frac{d^{n-N_0}}{3}\ln 2\geq d^{N_0}\ln (u_d^2D) \geq {d^{N_0}}\ln (|u_d|B_{ 0}D).\]

	If $\widehat{B_{0}}\geq 2$, by Lemma~\ref{first} and (\ref{Eq2}) again, for every $n\geq N_2$ we have \begin{multline*}
	\ln B_{n-N_0}\geq \frac{d^{n-N_0}}{3}\ln \widehat{B_{0}}\geq  \left(\frac{d^{n-N_0}}{3}-d^{N_0}\right)\ln 2+d^{N_0}\ln \widehat{B_{0}}\\
	\geq d^{N_0}\ln (u_d^2\widehat{B_{ 0}}D)\geq d^{N_0}\ln (|u_d|B_{ 0}D).
	\end{multline*}

	Therefore, there are most $W(d^{N_0},1/10)$ many integers $n\geq \max\{N_1+n',N_2+n'\}$ such that $g_c^{n-N_0}(c)$ is a rational solution to (\ref{Eq6}).
	Combined with 
	\begin{multline*}
	|g_c^n(c)|=|g_c^{N_0}\left(g_c^{n-N_0}(c)\right)|
	=\prod_{\xi\in \Gamma}|g_c^{n-N_0}(c)-\xi|
	=\prod_{\xi\in \Gamma_1}|g_c^{n-N_0}(c)-\xi|\cdot\prod_{\xi\in \Gamma_2}|g_c^{n-N_0}(c)-\xi|,
	\end{multline*}
	this implies that for all but at most $\#\Gamma_2\cdot W(d^{N_0},1/10)$ many $n\geq \max\{N_1+n',N_2+n'\}$ we have
	\begin{equation}\label{eq:k1}
	\ln |g_c^n(c)|\geq  \#\Gamma_1\cdot\ln \epsilon-3\cdot\# \Gamma_2\ln \left(2DB_{n-N_0}\right).
	\end{equation}
	Combining (\ref{eq:key0}) and our assumption $0<\epsilon<1$, we have
\[ \#\Gamma_1\cdot\ln \epsilon-3\cdot\# \Gamma_2\ln \left(2DB_{n-N_0}\right) 	\geq d^{N_0}\ln \left(\frac{\epsilon}{8D^3}\right)-3\#\Gamma_2\ln B_{n-N_0}.\]
	From our assumption that $g(x)\neq u_dx^d$ or $\alpha\neq 0$, we have \[\# \Gamma_2=m_{i_0}\leq (d-1)^{N_0}.\]
	Therefore, the previous statement implies that 
	 for all but at most \\
	 $(d-1)^{N_0} W(d^{N_0},1/10)$ many $n\geq \max\{N_1+n',N_2+n'\}$ we have
	 \begin{equation}\label{new2}
	 \ln |g_c^n(c)|
	 \geq d^{N_0}\ln \left(\frac{\epsilon}{8D^3}\right)-3(d-1)^{N_0}\ln B_{n-N_0}.
	 \end{equation}
	 By (\ref{Eq2}) and (\ref{eq:n0}), we have
	 \[3(d-1)^{N_0}\ln B_{n-N_0}\leq \frac{d^{N_0-1}(d-1)}{3}\ln B_{n-N_0}\leq \frac{d^{N_0-1}(d-1)}{3}\ln (|u_d|\widehat{B_{n-N_0}}).\]
	 By  Lemma~\ref{first}, we obtain
\begin{multline*}
	 \frac{d^{N_0-1}(d-1)}{3}\ln (|u_d|\widehat{B_{n-N_0}})\leq \frac{d^{N_0-1}(d-1)}{3}\ln |u_d|+\frac{d-1}{d}\ln B_{n}\\
	 \leq  d^{N_0}\ln |u_d|+\frac{d-1}{d}\ln B_{n}
\end{multline*}
	 The two inequality above implies that 
	  \[\textrm{Right Hand Side of~} (\ref{new2})\geq  d^{N_0}\ln \left(\frac{\epsilon}{8|u_d|D^3}\right)+(-1+1/d)\ln B_{n}, \]
	 which 
	 completes the proof.
\end{proof}

\section{Proof of Proposition~\ref{thm:fund ineq}}\label{s3}
The basic idea of proving Proposition~\ref{thm:fund ineq} is to show that for each $x^2$-divisible  $g(x)\in \ZZ[x]$ there exists a finite cover of $\SS_g$ as follows:
\begin{enumerate}
	   \renewcommand{\theenumi}{(\arabic{enumi})}
	\item $(-\infty,-4L_g]\cup[4L_g,\infty)$;
	\item $(\alpha-\delta_{g,\alpha},\alpha+\delta_{g,\alpha})\cap \UU_g$ for finitely many $\alpha$ in $[-4L_g,4L_g] $ with $0<\delta_{g,\alpha}<L_g$;
	\item the finite set $\UU_g^0$,
\end{enumerate}
such that $g$ has rapidly increasing numerators on every set in this cover.

Recall that for every $n\in \NN$ we denote by $\omega(n)$ the number of its distinct prime divisors. For convenience, we put $s_d(n)\coloneqq \sum_{p|n}d^{\frac{n}{p}}$. Then we have the following estimation.
\begin{lemma}\label{l:sk}
	For every $d\geq 2$ and every $n\geq 30$, we have $s_d(n)\leq d^{\frac{3n}{5}}$.
\end{lemma}
\begin{proof}
	
	For every integer $n\geq 2$ we have $n\geq 2^{\omega(n)}$ and hence 
	\begin{equation}\label{Eq10}
	\omega(n)\leq \log_2n.
	\end{equation}
	
	Since for every prime divisor $p$  of $n$ we have $n/p\leq n/ 2$. 
	Combined with (\ref{Eq10}), we have
	\begin{equation}\label{ineq:s_d}
	s_d(n)\leq d^{\frac{n}{2}}\omega(n)\leq   d^{\frac{n}{2}}\log_2n.
	\end{equation}
	
	On the other hand, for every $n\geq 30$ we have
	\[\log_2 n< 5< 2^3\leq  d^{\frac{n}{10}} .\]
	Together with (\ref{ineq:s_d}), this finishes the proof.
\end{proof}

\begin{lemma}\label{L:3}
	Given any $x^2$-divisible $g(x)\in \ZZ[x]$ and $c\in \QQ$, for every $n\geq 30$ we have
	\begin{equation*}
	\sum_{p|n}\ln \left|A_{n/p}\right|\leq d^{\frac{3n}{5}}\ln \Big(2|u_d|^2\widehat{B_{0}}\max\left\{|c|,4L_g\right\} \Big).
	\end{equation*}
\end{lemma}
\begin{proof}
	By Lemma~\ref{lem:Bn<}, for every $n\geq 0$ we have 
		\begin{equation*}
	\sum_{p|n}\ln \left|A_{n/p}\right|\leq \sum_{p|n} d^{n/p}\ln \Big(2|u_d| B_0 \max\left\{|c|,4L_g\right\}\Big).
	\end{equation*}
	Together with 
	Lemma~\ref{l:sk} and (\ref{Eq2}), this finishes the proof.
\end{proof}

\begin{proposition}\label{P:1}
	Every $x^2$-divisible  $g(x)\in \ZZ[x]$ of degree $d\geq 3$ has rapidly increasing numerators on $(-\infty,-4L_g]\cup[4L_g,\infty)$.
\end{proposition}
\begin{proof}
	Let $c$ be an arbitrary rational number in $(-\infty,-4L_g]\cup[4L_g,\infty)$.

	By Lemma~\ref{first}, Corollary~\ref{lemma f>} with $n'=0$ and $|c|\geq 4L_g\geq 4$, for every $n\geq 0$ we have
	\begin{equation}\label{eq:15}
	\ln |A_n|\geq d^{n}\ln |c/2|+{\frac{d^n}{3}}\ln \widehat{B_{0}}\geq {\frac{d^n}{3}}\ln |c\widehat{B_{0}}|.
	\end{equation}
	Combined with Lemma~\ref{L:3} and $\left|c\right|\geq 4L_g$,
	this implies that for every $n\geq 30$ we have
	\begin{equation*}
	\ln |A_n|-\sum_{p|n}\ln \left|A_{n/p}\right|
	\geq 
	-d^{\frac{3n}{5}}\ln (2|u_d|^2)+\left(d^n/3-d^{\frac{3n}{5}}\right)\ln \left|c\widehat{B_{0}}\right|.
	\end{equation*}

	Therefore, there exists an integer $N>\max\left\{\frac{5}{2}\log_d 3,30\right\}$, which only depends on $g$, such that for every $n\geq N$ and every rational number $c\in (-\infty,-4L_g]\cup[4L_g,\infty)$   we have \[\ln |A_n|>\sum_{p|n}\ln \left|A_{n/p}\right|.\]
	Thus we prove this proposition.
\end{proof}

We next prove the following. 
\begin{proposition}\label{P:2}
	Every polynomial $g(x)=u_dx^d\in \ZZ[x]$ of degree $d\geq 3$ has rapidly increasing numerators on $\left(-\frac{1}{|4u_d|},\frac{1}{|4u_d|}\right)\cap \SS_g$.
\end{proposition}
Since $0\notin \SS_g$, it is sufficient to show the following two lemmas.

\begin{lemma}\label{L:1}
	Every polynomial $g(x)=u_dx^d\in \ZZ[x]$ of degree $d\geq 3$ has rapidly increasing numerators on $\left(0, \frac{1}{4|u_d|}\right)\cap \SS_g$.
\end{lemma}

\begin{proof}
	For every $c\in (0, \frac{1}{4|u_d|})\cap\SS_g$ and $n \geq 0$ we have
	\begin{equation}\label{Eq8}|c| \leq |g_c^n(c)| \leq (|u_d|+1)^{\frac{d^{n}-1}{d-1}} |c|	\leq (|u_d|+1)^{d^n}|c|\end{equation}
	and 
	$\widehat{B_{0}}\geq 2.$
	
	Combining Lemmas~\ref{lem:Bn<}(1), \ref{first} for $n'=0$ with	(\ref{Eq8}), for every $n\geq 0$ we have
	\begin{equation*}
	\ln |A_n|-\sum_{p|n}\ln \left|A_{n/p}\right|
	\geq (1-\omega(n))\ln |c|+\frac{d^n}{3}\ln \widehat{B_{0}}-s_d(n)\ln\big((|u_d|+1)B_0\big).
	\end{equation*}
	
	Together with (\ref{Eq2}) and lemma~\ref{l:sk}, this implies that for every $n\geq 30$ we have
	\[\ln |A_n|-\sum_{p|n}\ln \left|A_{n/p}\right|	\geq -d^{\frac{3n}{5}} \ln\big((|u_d|+1)|u_d| \big)+\left({\frac{d^n}{3}}-d^{\frac{3n}{5}}\right)\ln \widehat{B_{0}}. \]

	From $\widehat{B_{0}}\geq 2$, there exists an integer $N\geq 30$ such that 
	for every $n\geq N$ we have 
	\[\ln |A_n|-\sum_{p|n}\ln \left|A_{n/p}\right|>0,\]
	which completes the proof.
\end{proof}
\begin{lemma}\label{L:2}
	Every polynomial $g(x)=u_dx^d\in \ZZ[x]$ of degree $d\geq 3$ has rapidly increasing numerators on $\left(-\frac{1}{4|u_d|},0\right)\cap \SS_g$.
\end{lemma}

\begin{proof}
	Note that when	$d$	is odd, we may replace	$c$	with	$-c$	and the forward orbit of $0$ will be unchanged, modulo sign.
	By Lemma~\ref{L:1}, we immediately prove this case.

	Therefore, it is sufficient to study the case that $d$ is even.
	We first show that for every $c\in (-\frac{1}{4|u_d|},0)$ and every $n\geq 0$  we have 
	\begin{equation}\label{Eq7}
	|c| (1-|u_d| |c|^{d-1}) \leq |g_c^n(c)| \leq |c|.
	\end{equation}
	
	For  $n=0$, we have $|g_c^0(c)|=|c|$.
	Assume that (\ref{Eq7}) holds for some $n\geq 0$. Since $g_c(x)$ is negative and decreasing on $(-\frac{1}{4|u_d|},0)$, we have
	\[		|c| (1-|u_d| |c|^{d-1})= |g_c(c)|\leq |g_c(g_c^{n}(c))|\leq |f(0)|=|c|,\]
	which proves (\ref{Eq7}) by induction.

	Combining Lemmas~\ref{lem:Bn<}(1), \ref{first} with 
	(\ref{Eq7}), we have
	\begin{multline*}
	\ln |A_n|-\sum_{p|n}\ln \left|A_{n/p}\right|
	\geq 
	\ln\Big(	|c| \cdot(1-|u_d| |c|^{d-1}) \Big)+ {\frac{d^n}{3}}  \ln \widehat{B_{0}} - s_d(n)\ln B_0-\omega(n)\ln |c|
	\\
	\geq\ln(1-|u_d| |c|^{d-1}) + {\frac{d^n}{3}}  \ln \widehat{B_{0}} - s_d(n)\ln B_0.
	\end{multline*}
	Together with (\ref{Eq2}), Lemma~\ref{l:sk} and $|c|< \frac{1}{4|u_d|}$, this implies that for every $n\geq 30$ we have
	\begin{equation}\label{Eq11}
	\ln |A_n|-\sum_{p|n}\ln \left|A_{n/p}\right|\geq \ln(3/4 )+d^{\frac{3n}{5}}\ln |u_d|+\left({\frac{d^n}{3}}-d^{\frac{3n}{5}}\right)\ln \widehat{B_{0}}.
	\end{equation}
	
	On the other hand, for every $c\in (-\frac{1}{4|u_d|},0)$, we have $\widehat{B_{0}}\geq 2$.
	Combined with (\ref{Eq11}),  this proves that there exists an integer $N\geq 30$ such that 
	for every $n\geq N$ we have 
	\[\ln |A_n|-\sum_{p|n}\ln \left|A_{n/p}\right|>0,\]
	which completes the proof.
\end{proof}

\begin{proposition}\label{P:3}
	Given any $x^2$-divisible   $g(x)\in \ZZ[x]$ of degree $d\geq 3$ and $\alpha\in [-4L_g,4L_g]$ such that $g(x)\neq u_dx^d$ or $\alpha\neq 0$, there is an $0<\delta<L_g$ such that $g$ has rapidly increasing numerators on $c\in (\alpha-\delta,\alpha+\delta)\cap \UU_g$.
\end{proposition}
\begin{proof}
	Note that for every $c\in \UU_g$ we have $\widehat{B_{0}}\geq 2$. By Proposition~\ref{p:1} with $n'=0$, there is a $3$-tuple 
	$0<\delta<L_g$, $C>0$ and $N_1>0$ such that for every $c\in (\alpha-\delta,\alpha+\delta)\cap \UU_g$ there is a finite set $S_c\subset \NN$ of bounded cardinality $N_1$ such that for every $n\not\in S_c$ we have
	\[\ln|g_{c}^{n}(c)|\geq  \min\left\{({-1+1/d})\ln B_n+\ln C,\ \ln \delta\right\},\]	 and therefore
	\begin{equation}\label{eq:17}
	\ln |A_n|\geq\min\left\{\frac{1}{d}\ln B_n+\ln C,\ \ln B_n+\ln \delta\right\}.
	\end{equation}

	On the other hand, by Lemma~\ref{first}, we have
	\begin{align*}
	\ln(B_n)+\ln\delta-d^{\frac{3n}{5}}\ln \left|2cu_d^2\widehat{B_{0}}\right|\geq& \ln \delta+\left(\frac{d^n}{3}-d^{\frac{3n}{5}}\right)\ln \widehat{B_{0}}-d^{\frac{3n}{5}}\ln |2cu_d^2|,\\
	\frac{1}{d}\ln( B_n)+\ln  C-d^{\frac{3n}{5}}\ln \left|2cu_d^2\widehat{B_{0}}\right|\geq & \ln C+\left(\frac{d^{n-1}}{3}-d^{\frac{3n}{5}}\right)\ln \widehat{B_{0}}-d^{\frac{3n}{5}}
	\ln |2cu_d^2|.
	\end{align*}

	Combined with Lemma~\ref{L:3}, (\ref{eq:17}) and $|c|\leq 5L_g$, this implies that there exists an integer $N_2\geq 30$ such that for every rational number $c\in (\alpha-\delta,\alpha+\delta)$  and $n\in \{N_2,N_2+1,\dots,\}\backslash S_c$ we have \[\ln |A_n|>\sum_{p|n}\ln \left|A_{n/p}\right|.\]
	Taking $N\coloneqq N_1+N_2$, we prove this proposition.
\end{proof}

Now we turn our attention to the finite set $\UU_g^0$. 
\begin{proposition}\label{P:4}
	Every $x^2$-divisible  $g(x)\in \ZZ[x]$ of degree $d\geq 3$ has rapidly increasing numerators on the finite set $\UU_g^0$.
\end{proposition}
\begin{proof}
	It is sufficient to show that for 
	each individual rational number in $\UU_g^0$ there are finite many $n\in\NN$  satisfying
	(\ref{eq:2}).
	
	Let $c$ be an arbitrary rational number in $\UU_g^0$. We first show that there must exist an integer $n'$ such that either $|g_c^{n'}(c)|>4L_g$ or $\widehat{B_{n'}}\geq 2$. Suppose that for every $n\geq 0$ we have $\widehat{B_{n}}= 1$, i.e., $B_n|u_{d}$. Since $c\in \SS_{g}$ and there are only finitely many integers in $[-4L_g,4L_g]$ with denominator dividing $u_d$, we know that there must exist an  $n'\geq 0$ such that $|g_c^{n'}(c)|>4L_g$. 
	
	(1) 
	When $|c|\leq 4L_g$, $\widehat{B_{0}}=1$ and there exists an integer $n'\geq 1$ such that 
	$|g_c^{n'}(c)|>4L_g$. Combining these conditions with Corollary~\ref{lemma f>} and Lemma~\ref{L:3}, 
	for every $n\geq \max\{30,n'\}$ we have
	\begin{equation*}
	\ln |A_n|-\sum_{p|n}\ln \left|A_{n/p}\right|
	\geq 
	-d^{\frac{3n}{5}}\ln (8L_g|u_d|^2)+d^{n-n'}\left|g_c^{n'}(c)/2\right|.
	\end{equation*}
	Clearly, there exists an integer $N>\max\{30,n'\}$ such that for every $n\geq N$ we have \[\ln |A_n|>\sum_{p|n}\ln \left|A_{n/p}\right|.\]
	
	(2)
	When $|c|\leq 4L_g$, $\widehat{B_{0}}=1$ and there is an integer $n'\geq 1$ such that 
	$\widehat{B_{n'}}\geq 2$. 
	Similar to Proposition~\ref{P:3}, 
	we combine Lemma~\ref{first} with Proposition~\ref{p:1}, and obtain a finite set $S_c\subset \NN$, an integer $N$, $\delta_c>0$ and $C_c>0$ such that for every $n\geq n'$ if $n\not\in S_c$, then
	\[\ln|g_{c}^{n}(c)|\geq \min\left\{({-1+1/d})\ln B_n+\ln C_c,\ \ln \delta_c\right\},\]	 and therefore
	\begin{equation}\label{Eq:12}
	\ln |A_n|\geq\min\left\{\frac{1}{d}\ln B_n+\ln C_c,\ \ln B_n+\ln \delta_c\right\}.
	\end{equation}

	On the other hand, by Lemma~\ref{first}, we have
	\begin{align*}
	\ln B_n+\ln \delta_c-d^{\frac{3n}{5}}\ln \left|2cu_d^2\widehat{B_{0}}\right|\geq& \ln\delta_c+\frac{d^{n-n'}}{3}\ln \widehat{B_{n'}}-d^{\frac{3n}{5}}\ln |2cu_d^2|,\\
	\frac{1}{d}\ln B_n+\ln C_c-d^{\frac{3n}{5}}\ln \left|2cu_d^2\widehat{B_{0}}\right|\geq&  \ln C_c+\frac{d^{n-n'-1}}{3}\ln \widehat{B_{n'}}-d^{\frac{3n}{5}}
	\ln |2cu_d^2|.
	\end{align*}

	Combined with Lemma~\ref{L:3}, (\ref{Eq:12}) and $|c|\leq 4L_g$, this implies that there exists an integer $N_1>\max\{n',30\}$ such that for every $n\in \{N_1,N_1+1,\dots\}\backslash S_c$ we have \[\ln |A_n|>\sum_{p|n}\ln \left|A_{n/p}\right|.\]
	Put $N\coloneqq \#S_c+N_1$. Then we complete the proof.
\end{proof}

\begin{proof}[Proof of Proposition~\ref{thm:fund ineq}]\label{s32}
	By Proposition~\ref{P:3}, for every $g(x)\neq u_dx^d$ or $\alpha\neq 0$ and every real number $\alpha\in [-4L_g,4L_g]$, there is an $0<\delta_{g,\alpha}<L_g$ such that $g$ has rapidly increasing numerators on $c\in (\alpha-\delta_{g,\alpha},\alpha+\delta_{g,\alpha})\cap \UU_g$.
	
	For $g(x)= u_dx^d$ and $\alpha=0$, if we put $\delta_{g,\alpha}\coloneqq \frac{1}{4|u_d|}$, then we proved in Proposition~\ref{P:2} that $g$ has rapidly increasing numerators on $c\in (-\delta_{g,\alpha},\delta_{g,\alpha})\cap \UU_g$.
	
	Now for every $x^2$-divisible  $g(x)\in \ZZ[x]$ we obtain a cover of $\SS_{g}$ as
	\begin{equation}\label{Eq12}
	(-\infty,-4L_g]\cup[4L_g,\infty)\cup \UU_g^0\cup \bigcup\limits_{\alpha\in [-4L_g,4L_g]}\Big((\alpha-\delta_{g,\alpha},\alpha+\delta_{g,\alpha})\cap \UU_g\Big).
	\end{equation}
	
	Note that \[\bigcup\limits_{\alpha\in [-4L_g,4L_g]} (\alpha-\delta_{g,\alpha},\alpha+\delta_{g,\alpha})\] is an open cover of the closed interval $[-4L_g,4L_g]$, which has a finite cover. We use the center $\alpha$ to represent the interval $(\alpha-\delta_{g,\alpha},\alpha+\delta_{g,\alpha})$ in this finite cover, and put $T$ to be the index set of $\alpha$. 
	
	Therefore, we obtain a finite cover of $\SS_g$ as follows:
	\[(-\infty,-4L_g]\cup[4L_g,\infty)\cup \UU_g^0\cup \bigcup\limits_{\alpha\in T}\Big((\alpha-\delta_{g,\alpha},\alpha+\delta_{g,\alpha})\cap \UU_g\Big).\]
	
	By Propositions~\ref{P:1}, \ref{P:2}, \ref{P:3} and \ref{P:4}, we know that $g$ has rapidly increasing numerators on each set in this cover. It implies that $g$ also has rapidly increasing numerators on $\SS_g$, which completes the proof.
\end{proof}
\section{Theorem~\ref{main thm 2} implies Theorem~\ref{main thm3}}

At the end of this section, we will prove that Theorem~\ref{main thm 2} implies the following proposition which leads to Theorem~\ref{main thm3}.

\begin{proposition}\label{main thm}
	For every $x^2$-divisible polynomial $g(x)\in \QQ[x]$
	there is a constant $\mathbf{M}_g>0$, depending only on $g$, such that
	\[\#\calZ(g_c,0)\leq \mathbf{M}_g\]  for every $c\in \SS_{g}$. 
\end{proposition}

For every $n\in \ZZ\backslash\{0\}$, we denote by $\omega(n)$ the number of distinct prime factors of $n$.
\begin{lemma}\label{l:1}
	Let $g(x)\in\QQ[x]$ be $x^2$-divisible  and $a\in \ZZ\backslash\{0\}$. Then with $h(x)\coloneqq \f{1}{a}g(ax)$ we have
	\[|\#\calZ(g_c,0)-\#\calZ(h_{c/a},0)|\leq \omega(a) \quad\textrm{for every~} c\in \QQ. \]
\end{lemma}

\begin{proof}	
	Consider that 
	\[
	g^n_c(ax)=uh^n_{c/a}(x)\quad\mbox{and}\quad g^n_c(0)=ah^n_{c/a}(0).
	\]
	For every $p\nmid u$ we have
	$p$ is a primitive prime divisor of $g^n_c(0)$  if and only if it is a primitive prime divisor of $h^n_{c/a}(0)$. 
	Therefore, the difference between $\#\calZ(g_c,0)$ and $\#\calZ(h_{c/a},0)$ can not exceed the number of prime factors of $a$, which completes the proof.
\end{proof}

\begin{lemma}\label{l:2}
	For any $x^2$-divisible $g(x)\in\QQ[x]$ and $t\in \QQ\backslash \{0\}$, 
	Proposition~\ref{main thm} holds for $g$ if and only if it holds for $\frac{1}{t}g(tx)$.
\end{lemma}
\begin{proof}
	Let 
	$t=\frac{a}{b}$ be an arbitrary rational number.
	Put \[h(x)\coloneqq \f{1}{t}g(tx)\quad\textrm{and}\quad h_1(x)\coloneqq\f{1}{a}g(ax).\]
	Note that $h_1(x)=\frac{1}{b}h(bx)$.
	
	By Lemma~\ref{l:1} for every $c\in \QQ$ we have
	\begin{align*}
	|\#\calZ(g_c,0)-\#\calZ((h_1)_{c/a},0)|&\leq \omega(a),\\
	|\#\calZ(h_{c/t},0)-\#\calZ((h_1)_{c/a},0)|&\leq \omega(b),
	\end{align*}
	which implies
	\[|\#\calZ(g_c,0)-\#\calZ(h_{c/t},0)|\leq \omega(a)+\omega(b).\]
	
	Since
	$\omega(a)$ and $\omega(b)$ are both independent of $c$, we complete the proof.
\end{proof}

%Therefore, to prove Theorem~\ref{main thm3}, it is enough to show the following.
%

\begin{proposition}\label{aaa}
	Theorem~\ref{main thm 2} implies Theorem~\ref{main thm3}.
\end{proposition}

\begin{proof}
\noindent\textbf{Step I:} We first prove that 
	Theorem~\ref{main thm 2} implies Proposition~\ref{main thm}.
	
	Given any $x^2$-divisible polynomial $g(x)\in \QQ[x]$ of degree $d\geq 2$, if $d=2$, then we can find $t\in \QQ$ such that $\frac{1}{t}g(tx)=x^2$. Combining Lemma~\ref{l:2} with \cite[Theorem 1.1]{Krieger2013}, we prove Proposition~\ref{main thm} for this case. 
	
	If $d\geq 3$, there is $t\in \ZZ$ such that $\frac{1}{t}g(tx)\in \ZZ[x]$. Combined with Lemma~\ref{l:2}, this proves Proposition~\ref{main thm} for this case.

\noindent\textbf{Step II:} Now we prove that Proposition~\ref{main thm}  implies Theorem~\ref{main thm3}.
	
	For every polynomial $g(x)\in \QQ[x]$ with a critical point $u\in \QQ$, if we put $f(x)\coloneqq g(x+u)-u$, then we know that $0$ is a critical point of $f$ and 
	for every $c\in \QQ$ we have $f_c^n(0)=g_c^n(u)-u$ and hence
	\[\calZ(g_c,u)=\calZ(f_c,0).\]
	This completes the proof. 
\end{proof}

\end{document}